\documentclass[a4paper]{amsart}
\def\doctype{}
\usepackage{latexsym,amssymb}
\usepackage{color}
\usepackage{dsfont}
\usepackage{fancyhdr}
\usepackage{nicematrix}
\usepackage{tikz}
\usepackage{hyperref}
\usepackage{float}
\usepackage{enumitem}

\renewcommand\S{\mathrm{S}}

\newcommand{\comment}[1]{}

\fancypagestyle{titlepage}{
\fancyhead[R]{\doctype}
\fancyhead[CL]{}
\cfoot{\vspace{5pt} \thepage}
}

\let\oldsection\section
\newcommand\boldsection[1]{\oldsection{\bf #1}}
\newcommand\starsection[1]{\oldsection*{\bf #1}}
\makeatletter
\renewcommand\section{\@ifstar\starsection\boldsection}
\makeatother

\newtheoremstyle{theorem}
  {12pt}		  
  {0pt}  
  {\sl}  
  {\parindent}     
  {\bf}  
  {. }    
  { }    
  {}     
\theoremstyle{theorem}
\newtheorem{thm}{Theorem}[section]  
\newtheorem{lemma}[thm]{Lemma}     
\newtheorem{cor}[thm]{Corollary}

\newtheorem{prop}[thm]{Proposition}

\newtheoremstyle{definition}
  {12pt}		  
  {0pt}  
  {}  
  {\parindent}     
  {\bf}  
  {. }    
  { }    
  {}     
\theoremstyle{definition}

\newtheorem{ex}[thm]{Example}

\renewcommand{\proofname}{Proof}

\makeatletter
\renewenvironment{proof}[1][\proofname]{\par
  \pushQED{\qed}%
  \normalfont \partopsep=\z@skip \topsep=\z@skip
  \trivlist
  \item[\hskip\labelsep
        \scshape
    #1\@addpunct{.}]\ignorespaces
}{%
  \popQED\endtrivlist\@endpefalse
}
\makeatother

\makeatletter
\renewcommand*\@maketitle{%
  \normalfont\normalsize
  \@adminfootnotes
  \@mkboth{\@nx\shortauthors}{\@nx\shorttitle}%
  \global\topskip1\p@\relax 
  \@settitle
  \ifx\@empty\authors \else {\vskip 0em
\vtop{\centering\shortauthors\@@par}} \fi
  \ifx\@empty\@date \else {\vskip 1em \vtop{\centering\@date\@@par}}\fi 
  \ifx\@empty\@dedicatory
  \else
    \baselineskip3\p@
    \vtop{\centering{\footnotesize\itshape\@dedicatory\@@par}%
      \global\dimen@i\prevdepth}\prevdepth\dimen@i
  \fi
  \@setabstract
  \normalsize
  \if@titlepage
    \newpage
  \else
    \dimen@20\p@ \advance\dimen@-\baselineskip
    \vskip\dimen@\relax
  \fi
} 
\renewcommand*\@adminfootnotes{%
  \let\@makefnmark\relax  \let\@thefnmark\relax
  \ifx\@empty\@subjclass\else \@footnotetext{\@setsubjclass}\fi
  \ifx\@empty\@keywords\else \@footnotetext{\@setkeywords}\fi
  \ifx\@empty\thankses\else \@footnotetext{%
    \def\par{\let\par\@par}\@setthanks}%
  \fi
\thispagestyle{titlepage}
}
\makeatother

\pgfsetlayers{nodelayer,edgelayer} 
\usepackage{graphicx}

\begin{document}

\title[SymAlt]{\large The base size of the symmetric group \\acting on subsets}

\author{Coen del Valle and Colva M. Roney-Dougal}
\address{
School of Mathematics and Statistics,
University of St Andrews, St Andrews, UK
}
\email{cdv1@st-andrews.ac.uk, Colva.Roney-Dougal@st-andrews.ac.uk}

\thanks{Research of Coen del Valle is supported by the Natural Sciences and Engineering Research Council of Canada (NSERC), [funding reference number PGSD-577816-2023], as well as a University of St Andrews School of Mathematics and Statistics Scholarship.}
\date{\today}

\begin{abstract}
A base for a permutation group $G$ acting on a set $\Omega$ is a subset $\mathcal{B}$ of $\Omega$ such that the pointwise stabiliser $G_{(\mathcal{B})}$ is trivial. Let $n$ and $r$ be positive integers with $n>2r$. The symmetric and alternating groups $\mathrm{S}_n$ and $\mathrm{A}_n$ admit natural primitive actions on the set of $r$-element subsets of $\{1,2,\dots, n\}$. Building on work of Halasi ~\cite{hal}, we provide explicit expressions for the base sizes of all of these actions, and hence determine the base size of all primitive actions of $\S_n$ and $\mathrm{A}_n$.
\end{abstract}

\maketitle
\hrule

\bigskip

\section{Introduction}

A \emph{base} for a permutation group $G$ acting on a set $\Omega$ is a subset $\{\alpha_1,\alpha_2,\dots, \alpha_k\}\subseteq\Omega$ with trivial pointwise stabiliser in $G$. Bases have proved to be tremendously useful in permutation group algorithms (see e.g. ~\cite{ser}); for many computations the complexity is a function of the size of the base used, so it is of interest to find a smallest possible base. The size $b(G)$ of a smallest base for $G$ is called the \emph{base size} of $G$. Blaha ~\cite{blaha} shows that for any positive integer $l$, and any permutation group $G$, the problem of determining whether $G$ admits a base of size at most $l$ is NP-complete. On the other hand, there are many different estimates known for $b(G)$ --- for example we can derive elementary upper and lower bounds as follows. If $\{\alpha_i\}_{i=1}^{k}$ is a base for $G$ then any group element $g\in G$ is uniquely determined by the tuple $(\alpha_i^g)_{i=1}^{k}$ and so $|G|\leq |\Omega|^{k}$. If $k=b(G)$ then $|G_{(\alpha_1,\alpha_2,\dots,\alpha_i)} : G_{(\alpha_1,\alpha_2,\dots,\alpha_{i+1})}|\geq 2$ for all $1\leq i<k$, so $2^{b(G)}\leq |G|$, and hence $(\log |G|)/(\log |\Omega|)\leq b(G)\leq \log |G|$.

In this paper we consider the primitive faithful actions of the symmetric and alternating groups. A result of Liebeck and Shalev ~\cite{lie3} states that there is some absolute constant $c$ such that for $G\in\{\S_n,\mathrm{A}_n\}$ acting primitively, either $b(G)\leq c$ or up to equivalence $G$ is acting on either
\begin{itemize}
    \item[(i)] $r$-subsets of $[n]:=\{1,2,\dots, n\}$ with $2r<n$; or
    \item[(ii)] partitions of $[kl]$ into $k$ parts of size $l$ with $kl=n$.
\end{itemize}
Such actions are called \emph{standard}, and all other primitive actions of $\S_n$ and $\mathrm{A}_n$ are \emph{non-standard}; we denote the permutation groups in case (i) by $\S_{n,r}$ and $\mathrm{A}_{n,r}$, respectively.

In fact, in ~\cite{lie3} they show a stronger result known as the Cameron-Kantor conjecture ~\cite{ck} that any almost simple primitive group either has a base of size at most $c$ (it has since been shown that  $c=7$ is best possible ~\cite[Corollary 1]{burn1}) or falls into one of three classes of exceptions including the two standard actions of $\S_n$ and $\mathrm{A}_n$. In recent years, Burness, Guralnick, and Saxl ~\cite[Corollaries 4 and 5]{burn2} showed that all non-standard actions of $\S_n$ and $\mathrm{A}_n$ have base size two or three, and ~\cite[Theorems 1.1 and 1.2]{ms} gave explicit formulae for all $k,l$ pairs in the actions on partitions, leaving only the actions on subsets to consider. 

A 2012 paper of Halasi ~\cite{hal} made progress on the subset action, showing that $b(\S_{n,r})\geq\left\lceil\frac{2n-2}{r+1}\right\rceil$ with equality when $n\geq r^2$, leaving only small $n$ and the action of the alternating group to consider. It is also shown in ~\cite{hal} that $b(\S_{n,r})\geq \lceil \log_2 n\rceil$ for all $n\geq 2r$ with equality if $n=2r$ (at which point the action is imprimitive).

In this paper we completely determine $b(\S_{n,r})$ and $b(\mathrm{A}_{n,r})$. Given $l,k,r\in\mathbb{N}$ set $m_r(l,k):=\frac{1}{k}\left(lr-\sum_{i=1}^{k-1}i{l\choose i}\right)$. Our main result is the following.

\begin{thm}\label{minbase}
Let $n\geq 2r$ be fixed and let $l$ be minimal such that there exists some $k\leq l+1$ satisfying $0\leq m_r(l,k)\leq{l\choose k}$ and $\sum_{i=0}^{k-1}{l\choose i}+m_r(l,k)\geq n.$ Then $b(\S_{n,r})=b(\mathrm{A}_{n+1,r})=l$.
\end{thm}

\begin{rks}
A pair $(l,k)$ satisfying the conditions of Theorem ~\ref{minbase} can be seen to exist by setting $l=n$ and $k=2$. We give a natural description of the quantity $m_r(l,k)$ at the beginning of Section ~\ref{minb}. In the case of the symmetric group, a similar result to Theorem ~\ref{minbase} was very recently determined independently by Mecenero and Spiga ~\cite{MeSp} --- their formula takes a different form and the proof is quite different.
\end{rks}

Putting together Theorem ~\ref{minbase}, ~\cite[Corollaries 4 and 5]{burn2}, and ~\cite[Theorems 1.1 and 1.2]{ms} the base size of all primitive actions of $\S_n$ and $\mathrm{A}_n$ are now known.

\begin{cor}
All almost simple primitive groups with alternating socle have known base size.
\end{cor}

The complicated statement of Theorem ~\ref{minbase} is unsurprising, as it needs to interpolate between $\lceil \log_2 n\rceil$ and $\left\lceil\frac{2n-2}{r+1}\right\rceil$. However, when restricting to specific functions $n$ of $r$, the result can be greatly simplified. Indeed, we deduce Corollary ~\ref{newlow}, which states that $b(\S_{n,r})=\left\lceil\frac{2n-2}{r+1}\right\rceil$ whenever $n\geq (r^2+r)/2$ --- a small extension of Halasi's result ~\cite[Theorem 3.2]{hal}. Furthermore, Corollary ~\ref{cor2} gives an explicit formula for $n$ at least roughly $r^{3/2}$.

The structure of the paper is as follows. In Section ~\ref{mach} we set up some general combinatorial machinery, establishing a connection between bases for $\S_{n,r}$ and a class of hypergraphs. In Section ~\ref{minb} we use the tools from Section ~\ref{mach} to prove Theorem ~\ref{minbase}, which we then use to obtain explicit formulae precisely determining the base size for specific functions $n$ of $r$.

\section{Bases and hypergraphs}\label{mach}
In this section we translate our problem into the language of hypergraphs. Define $\mathrm{S}_{n,\leq r}$ to be the symmetric group $\mathrm{S}_n$ acting in the natural way on the set of subsets of $[n]$ of size at most $r$. Halasi ~\cite{hal} shows that to determine $b(\S_{n,r})$ one may construct a minimum base for $\mathrm{S}_{n,\leq r}$.

\begin{lemma}[\cite{hal}]\label{hallem}
Fix $n\geq 2r$. Then $b(\S_{n,r})=b(\S_{n,\leq r})$.
\end{lemma}

To construct a minimum base for $\S_{n,\leq r}$ we start with a correspondence lemma which translates our bases into hypergraphs. Given a hypergraph $H=(V,E)$ and $v\in V$, define the $H$-\emph{neighbourhood} (or \emph{neighbourhood}, when clear from context) of $v$ to be the multiset $$N_H(v):=\{e\in E(H) : v\in e\},$$ and the \emph{degree} of $v$ to be the multiset size $|N_{H}(V)|$. 

Suppose $\mathcal{A}$ is a collection of subsets of $[n]$, each of size at most $r$. If there are two points $x,y\in [n]$ which are contained in all of the same sets in $\mathcal{A}$, then each element of $\mathcal{A}$ is fixed by the transposition $(x\, y)$, hence $\mathcal{A}$ is not a base for $\S_{n,\leq r}$. Similarly if no two such points exist then $\mathcal{A}$ is a base. We call a hypergraph \emph{irrepeating} if all hyperedges are distinct and all vertices have distinct neighbourhoods, so that the collections of edges and neighbourhoods form sets. 
Therefore, a collection $\mathcal{B}$ of distinct subsets of $[n]$ each of size at most $r$ is a base for $\S_{n,r}$ if and only if the pair $([n],\mathcal{B})$ forms an irrepeating hypergraph. We will often take the view that bases \emph{are} hypergraphs, and hence refer to the hypergraph $([n],\mathcal{B})$ simply as $\mathcal{B}$.

If a hypergraph $H$ is irrepeating and has $l$ vertices, $n$ hyperedges (including possibly the empty edge), and maximum vertex degree at most $r$ then we call $H$ an \emph{$(l,n,r)$-hypergraph}. We call two bases $\mathcal{B}_1,\mathcal{B}_2$ for $\S_{n,\leq r}$ \emph{equivalent} if there exists some $\sigma\in\S_{n,\leq r}$ with $\mathcal{B}_1^\sigma=\mathcal{B}_2$.

\begin{prop}\label{corr}
Fix positive integers $l,n$, and $r$, and let $\mathcal{L}$ be the set of isomorphism classes of $(l,n,r)$-hypergraphs and $\mathcal{S}$ be the set of all equivalence classes of bases of $\S_{n,\leq r}$ of size $l$. Then there exists a one-to-one correspondence $\rho: \mathcal{L}\to\mathcal{S}$.
\end{prop}

The correspondence $\rho$ is via a combinatorial construction known as the \emph{dual hypergraph}. Let $H$ be an irrepeating hypergraph. The \emph{dual} of $H$, denoted $H^\perp$, is the hypergraph with vertex set identified with the hyperedges of $H$, and hyperedges identified with vertices of $H$, where the incidence relations of $H^\perp$ are the reverse of those of $H$. That is $$V(H^\perp):=\{v_f : f\in E(H)\},$$ and $$E(H^\perp):=\{e_u : u\in V(H), v_{f}\in e_u\iff f\in N_H(u)\}.$$ The proof of Proposition ~\ref{corr} requires a couple of easy facts on the operation $\cdot^\perp$ given in the following lemma, which follows directly from the definition.

\begin{lemma}\label{dual}
Let $\mathcal{H}$ be the space of isomorphism classes of irrepeating hypergraphs. Then $\cdot^\perp$ is an involution in $\mathrm{Sym}(\mathcal{H})$. Moreover, $|V(H^\perp)|=|E(H)|$ and $|E(H^\perp)|=|V(H)|$ for all $H\in\mathcal{H}$.
\end{lemma}

\begin{proof}[Proof of Proposition ~\ref{corr}]
Let $\mathcal{B}$ be a base for $\S_{n,\leq r}$ of size $l$. By Lemma ~\ref{dual}, since $\mathcal{B}$ is irrepeating as a hypergraph it has an irrepeating dual, $H$, say, with $l$ vertices, $n$ edges, and maximum vertex degree at most $r$. That is, $H$ is an $(l,n,r)$-hypergraph. 

On the other hand, given an $(l,n,r)$-hypergraph, $K$, we deduce from Lemma ~\ref{dual} that $|V(K^\perp)|=n$, $|E(K^\perp)|=l$, and the largest edge of $K^\perp$ has size at most $r$. Moreover, Lemma ~\ref{dual} establishes that $K^\perp$ is irrepeating and so after relabelling the vertices of $K^\perp$ as $1,2,\dots,n$, the edges of the resulting hypergraph are indeed a base for $\S_{n,\leq r}$ of size $l$. Therefore, by Lemma ~\ref{dual} if we define $\rho$ to return the edge set of the composition of the dual operation, $\cdot^\perp$, together with any such relabelling of vertices, then $\rho:\mathcal{L}\to\mathcal{S}$ is a bijection.
\end{proof}


It now follows that to determine the base size of $\S_{n,\leq r}$ --- and hence, by Lemma ~\ref{hallem}, $\S_{n,r}$ --- it suffices to determine the minimum number of vertices of an irrepeating hypergraph with $n$ edges and maximum degree at most $r$. 

A hypergraph is called \emph{$k$-uniform} if its edge set consists only of edges of size $k$. A $k$-uniform hypergraph is called \emph{nearly-regular} if its degree sequence $a_1\geq a_2\geq\cdots\geq a_l$ satisfies $a_1-a_l\leq 1$. We use a result of Behrens et al.~\cite{beh} to prove the final lemma of this section.

\begin{lemma}\label{kexist}
Let $k,l,$ and $s$ be positive integers with $s\leq{l\choose k}$. Then there exists a nearly-regular $k$-uniform hypergraph on $l$ vertices with $s$ edges and highest degree $\left\lceil\frac{k}{l}s\right\rceil$.
\end{lemma}

\begin{proof}
If $l$ divides $ks$ then set $d=l$, otherwise let $d$ be the unique nonnegative integer such that $d\left\lceil\frac{k}{l}s\right\rceil+(l-d)\left\lfloor\frac{k}{l}s\right\rfloor=ks$. Consider the sequence $(a_1,a_2,\dots,a_l)$ where $$a_i=\begin{cases}\left\lceil\frac{k}{l}s\right\rceil&\text{ for $1\leq i\leq d$}\\\left\lfloor\frac{k}{l}s\right\rfloor&\text{ for $d+1\leq i\leq l$}.\end{cases}$$ From ${l\choose k}\geq s$, we deduce ${l-1\choose k-1}=\frac{k}{l}{l\choose k}\geq\frac{k}{l}s$, and hence $${l-1\choose k-1}=\left\lceil{l-1\choose k-1}\right\rceil\geq \left\lceil\frac{k}{l}s\right\rceil =a_1.$$ Therefore, by ~\cite[Theorem 2.1]{beh}, there exists a $k$-uniform hypergraph $H$ with degree sequence $(a_1,a_2,\dots, a_l)$, so $H$ is nearly-regular with $s$ edges.
\end{proof}

With duality in mind, to construct a small base for $\S_{n,\leq r}$ it suffices to build an irrepeating hypergraph with some fixed number of edges, but neighbourhoods as small as possible. A natural way to do this is to succesively add a smallest possible edge (in terms of set size), whilst ensuring no two vertices end up with the same neighbourhood. This is precisely how the main result of this section works. Recall $m_r(l,k)=\frac{1}{k}\left(lr-\sum_{i=1}^{k-1}i{l\choose i}\right)$ --- when clear from context, we omit the subscript $r$.

\begin{prop}\label{hypexist}
Fix positive integers $l,n$, and $r$ with $n\geq 2r$. Suppose there exists some $k\leq l+1$ such that $0\leq m(l,k)\leq{l\choose k}$, and $\sum_{i=0}^{k-1}{l\choose i}+m(l,k)\geq n.$
Then there exists an $(l,n,r)$-hypergraph.
\end{prop}

\begin{proof}
We construct such a hypergraph. Let $H_1$ be the unique irrepeating hypergraph on $l$ vertices with all possible edges of size at most $k-1$. 
By Lemma ~\ref{kexist}, since $\lfloor m(l,k)\rfloor\leq {l\choose k}$ there exists some $k$-uniform hypergraph $H_2$ on $l$ vertices with exactly $\lfloor m(l,k)\rfloor$ edges, and highest degree at most $$\left\lceil \frac{k}{l}\left\lfloor \frac{1}{k}\left(lr-\sum_{i=1}^{k-1}i{l\choose i}\right)\right\rfloor\right\rceil \leq \left\lceil \frac{k}{l}\left( \frac{1}{k}\left(lr-\sum_{i=1}^{k-1}i{l\choose i}\right)\right)\right\rceil = r-\sum_{i=1}^{k-1}{l-1\choose i-1}.$$ Let $H$ be the irrepeating hypergraph obtained by adding the edges of $H_2$ to $H_1$. Then $H$ has $l$ vertices, exactly $\sum_{i=0}^{k-1}{l\choose i}+m(l,k)\geq n$ edges, and highest degree at most $$\sum_{i=1}^{k-1}{l-1\choose i-1} +\left(r-\sum_{i=1}^{k-1}{l-1\choose i-1}\right)=r.$$ 
By arbitrarily deleting edges of size $k$ until exactly $n$ edges remain we do not increase the degree of any vertex, hence we obtain an $(l,n,r)$-hypergraph.
\end{proof}

\section{Base size of $\S_{n,r}$ and $\mathrm{A}_{n,r}$}\label{minb}

In this section we use the tools developed in Section ~\ref{mach} to deduce Theorem ~\ref{minbase}. We then illustrate how the construction works in practice with a brief example, before proving a couple of corollaries. 

We first give a description of the quantity $m_r(l,k)$ in terms of bases for $\S_{n,r}$ --- it is useful to start by stating the following lemma.

\begin{lemma}\label{doubcount}
Let $n$ and $r$ be positive integers with $n\geq 2r$, and $\mathcal{B}$ a base for $\S_{n,r}$ with $|\mathcal{B}|=l$. Then $lr=\sum_{x\in [n]}|N_{\mathcal{B}}(x)|$.
\end{lemma}

\begin{proof}
Count pairs $(B,x)$ where $x\in B\in\mathcal{B}$ in two ways.
\end{proof}

Given a base $\mathcal{B}$ as in the statement of Lemma ~\ref{doubcount} and some positive integer $k$, set $A_1:=\{x\in [n] : |N_{\mathcal{B}}(x)|< k\}$ and $A_2:=\{x\in [n] : |N_{\mathcal{B}}(x)|\geq k\}$. We can rewrite the sum in Lemma ~\ref{doubcount} as $lr=\left(\sum_{x\in A_1}|N_{\mathcal{B}}(x)|\right)+\left(\sum_{x\in A_2}|N_{\mathcal{B}}(x)|\right)$, hence $\sum_{x\in A_2}|N_{\mathcal{B}}(x)|=lr-\left(\sum_{x\in A_1}|N_{\mathcal{B}}(x)|\right)$. Since bases are irrepeating, it follows that there are at most ${l\choose i}$ distinct neighbourhoods of size $i$, hence $lr-\left(\sum_{x\in A_1}|N_{\mathcal{B}}(x)|\right)$ is minimally $lr-\sum_{i=1}^{k-1}i{l\choose i}=km_r(l,k)$. On the other hand $k|A_2|\leq\sum_{x\in A_2}|N_{\mathcal{B}}(x)|$. Therefore $m_r(l,k)$ estimates the minimum number of points of $[n]$ which have $\mathcal{B}$-neighbourhoods of size at least $k$.

We now proceed with the proof of the symmetric group case of Theorem ~\ref{minbase}.

\begin{proof}[Proof of Theorem ~\ref{minbase} for $\S_n$]
Let $\mathcal{B}$ be a minimum base for $\S_{n,r}$ so that $b:=b(\S_{n,r})=|\mathcal{B}|$. Let $l$ be minimal satisfying the conditions of Proposition ~\ref{hypexist}. Then there exists an $(l,n,r)$-hypergraph, $H$, say, and by Proposition ~\ref{corr}, $\rho(H)$ is a base of size $l$ for $\S_{n,\leq r}$. By the minimality of $b$ and Lemma ~\ref{hallem} we deduce $l\geq b$. 

Now, let $h\leq b+1$ be maximal such that $0\leq\left(br-\sum_{i=1}^{h-1}i{b\choose i}\right)$ (note $h$ exists since the inequality holds with $h=1<b+1$). It follows from the maximality of $h$ that $$0\leq\frac{1}{h}\left(br-\sum_{i=1}^{h-1}i{b\choose i}\right)\leq{b\choose h}.$$ From Lemma ~\ref{doubcount} and the subsequent discussion we deduce that $$br\geq\left(\sum_{i=0}^{h-1}i{b\choose i}\right)+h\left(n-\sum_{i=0}^{h-1}{b\choose i}\right),$$ where the second summand describes the fact that any point not contributing to the first summand has neighbourhood of size at least $h$. Rearranging gives $$\frac{1}{h}\left(br-\sum_{i=0}^{h-1}i{b\choose i}\right)\geq n-\sum_{i=0}^{h-1}{b\choose i}$$ and so $\sum_{i=0}^{h-1}{b\choose i}+m(b,h)\geq n$. Thus all conditions of Proposition ~\ref{hypexist} are satisfied. By definition $l$ is the smallest positive integer satisfying the conditions of Proposition ~\ref{hypexist} and so $l\leq b$, hence $l=b$ as desired.
\end{proof}

We now consider the alternating case.

\begin{proof}[Proof of Theorem ~\ref{minbase} for $\mathrm{A}_n$]
We prove the result by showing that $b(\mathrm{A}_{n+1,r})=b(\mathrm{S}_{n,r}).$ Let $\mathcal{B}$ be a base for $\S_{n,r}$ of size $b(\S_{n,r})$. Consider $\mathcal{B}$ as a collection of $r$-subsets of $[n+1]$, with $n+1$ having empty neighbourhood. Since $\mathcal{B}$ is a base for $\S_{n,r}$ at most one element of $[n]$ has an empty $\mathcal{B}$-neighbourhood, hence at most two elements of $[n+1]$ do, with all others distinct. But being a base for $\mathrm{A}_{n+1,r}$ is equivalent to having at most two points with equal neighbourhoods, hence $\mathcal{B}$ is a base for $\mathrm{A}_{n+1,r}$ of size $b(\S_{n,r})$. This shows that $b(\mathrm{A}_{n+1,r})\leq b(\S_{n,r})$.

On the other hand, let $\mathcal{C}$ be a base for $\mathrm{A}_{n+1,r}$. If there are two points $x$ and $y$ with the same $\mathcal{C}$-neighbourhood, then assume without loss of generality that $x=n+1$. Let $\mathcal{C}'$ be obtained from $\mathcal{C}$ by deleting $n+1$ from all $r$-sets in $\mathcal{C}$. Then $\mathcal{C}'$ is a collection of subsets of $[n]$ of size at most $r$ such that each element of $[n]$ has a distinct neighbourhood, that is, a base for $\S_{n,\leq r}$. Thus $b(\mathrm{A}_{n+1,r})\geq b(\S_{n,\leq r})$, and the result follows from Lemma ~\ref{hallem}.
\end{proof}

One can use the proof of the $\S_n$ case of Theorem ~\ref{minbase}, together with the proof of ~\cite[Theorem 2.1]{hal} to construct a minimum base for $\S_{n,r}$.

\begin{ex}\label{hypex}
Suppose we wish to construct a minimum base for $\S_{18,7}$. The pair $(l,k)=(5,3)$ satisfies the conditions of Theorem ~\ref{hypexist} with $l$ minimal. Therefore we start by taking the complete graph $K_5$ adorned with all loops and the empty edge. Following the procedure of the proof we then find a nearly-regular 3-uniform hypergraph on five points with highest degree at most $2$, and exactly two edges. In this case any simple 3-uniform hypergraph with exactly two edges will work, so we may pick one arbitrarily. At this point we have obtained the $(5,18,7)$-hypergraph in Figure ~\ref{hypfig}.

\begin{figure}[H]
  \includegraphics[width=4cm]{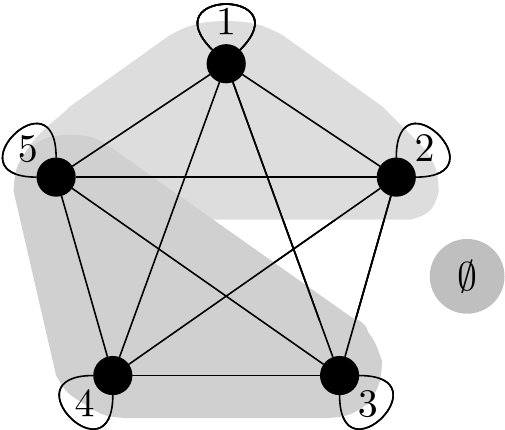}
\caption{The $(5,18,7)$-hypergraph constructed in Example ~\ref{hypex}, with vertices labelled arbitrarily.}\label{hypfig}
\end{figure}

Taking the dual and relabelling the vertices as $1,2,\dots, 18$ using the simplicial ordering gives $$\begin{array}{l}\{\{2,7,8,9,10,17\},\{3,7,11,12,13,17\},\{4,8,11,14,15,18\},\\\{5,9,12,14,16,18\},\{6,10,13,15,16,17,18\}\}\end{array}$$ as a minimum base for $\S_{18,\leq 7}$ --- one can now use Halasi's algorithm in ~\cite{hal} to transform this into a minimum base for $\S_{18,7}$.
\end{ex}

In ~\cite{hal} Halasi shows that $b(\S_{n,r})\geq\left\lceil\frac{2n-2}{r+1}\right\rceil$ with equality when $n\geq r^2$; using Theorem ~\ref{minbase} we can extend this result.

\begin{cor}\label{newlow}
Let $n$ and $r$ be positive integers $n\geq (r^2+r)/2$. Then $b(\S_{n,r})=\left\lceil\frac{2n-2}{r+1}\right\rceil$.
\end{cor}

\begin{proof}
First let $n\geq (r^2+r+2)/2$. Then $\left\lceil\frac{2n-2}{r+1}\right\rceil\geq\frac{2n-2}{r+1}\geq r$. Setting $k=2$ and $l=\left\lceil\frac{2n-2}{r+1}\right\rceil$ gives \begin{equation}\label{eq1}m(l,k)=\frac{1}{2}\left(lr-l\right)=\frac{1}{2}\left(\left\lceil\frac{2n-2}{r+1}\right\rceil (r-1)\right)\end{equation} and $$0\leq \frac{1}{2}\left(\left\lceil\frac{2n-2}{r+1}\right\rceil (r-1)\right)\leq \frac{1}{2}\left(\left\lceil\frac{2n-2}{r+1}\right\rceil \left(\left\lceil\frac{2n-2}{r+1}\right\rceil-1\right)\right)={l\choose k}.$$ Moreover, \begin{equation}\nonumber\begin{split}\sum_{i=0}^{k-1}{l\choose i}+m(l,k)&=1+\left\lceil\frac{2n-2}{r+1}\right\rceil+\frac{1}{2}\left(\left\lceil\frac{2n-2}{r+1}\right\rceil (r-1)\right)\quad\quad\quad\text{by (\ref{eq1})}\\&\geq 1+\frac{2}{r+1}(n-1)+\frac{r-1}{r+1}(n-1)\\&=n.\end{split}\end{equation} Therefore, $l$ satisfies the conditions of Theorem ~\ref{minbase}, and hence $$b(\S_{n,r})\leq \left\lceil\frac{2n-2}{r+1}\right\rceil.$$ Finally, one can construct a base for $n=(r^2+r)/2$ as follows. Our construction yields a base for $\S_{n+1,\leq r}$ with exactly one point with empty neighbourhood --- by deleting this point and relabelling if necessary we get a base for  $\S_{n,\leq r}$ of the desired size. We deduce equality from Halasi's lower bound ~\cite[Theorem 3.2]{hal}.
\end{proof}

We can continue to use Corollary ~\ref{minbase} to push even further down, into a range in which no formulae were previously known, although the formula is less pleasant.

\begin{cor}\label{cor2}
Let $n$ and $r$ be positive integers satisfying $\frac{r^2+r}{2}>n\geq r^{3/2}+\frac{r}{2}+1$. Then $$b(\S_{n,r})=\left\lceil\left(3\left(2n+r-\frac{5}{4}\right)+r^2\right)^\frac{1}{2}-r-\frac{3}{2}\right\rceil.$$
\end{cor}

\begin{proof}
Suppose first that there is some base $\mathcal{B}=\{B_1,\dots, B_k\}$ for $\S_{n,r}$ with largest vertex neighbourhood of size at most 2. Then $B_2$ has at least $r-1$ points not in $B_1$, $B_3$ has at least $r-2$ points not in $B_1\cup B_2$, and so on. Thus $n\geq r+(r-1)+\cdots+1=(r^2+r)/2$, a contradiction. Therefore, there is no base for $\S_{n,r}$ with largest neighbourhood size at most 2. Therefore if $\mathcal{B}$ is any base for $\S_{n,r}$ then at least one point has a neighbourhood of size at least three, thus by the discussion following Lemma ~\ref{doubcount}, $1+b+{b\choose 2}+m(b,3)\geq n$. That is, $$n\leq1+b+{b\choose 2}+\frac{1}{3}(br-b-b(b-1))=\frac{b^2+(2r+3)b}{6}+1,$$ solving for $b$ (via the quadratic formula, e.g.) gives $$b\geq\left(3\left(2n+r-\frac{5}{4}\right)+r^2\right)^\frac{1}{2}-r-\frac{3}{2}.$$
The above also shows that if we set $l$ to be (the ceiling of) the quantity above and $k=3$, then $l$ is the minimum positive integer satisfying $1+l+{l\choose 2}+m(l,k)\geq n$. Therefore, by Theorem ~\ref{minbase} if we can show $0\leq m(l,k)=\frac{1}{3}(lr-l^2)\leq {l\choose 3}$, or equivalently that $l\leq r\leq (l^2-l)/2+1$, then $b(\S_{n,r})=l$.

First, $l>r$ implies $(24n+4r^2+12r-15)^{1/2}>4r+3$. Rearranging gives $n>(r^2+r+2)/2$, a contradiction, so $l\leq r$. Finally, since $n\geq  r^{3/2}+\frac{r}{2}+1$, a straight-forward calculation shows that $r\leq (l^2-l)/2+1$, hence the result.
\end{proof}

\begin{rk}
In fact, the result holds for $(r^2+r)/2>n\geq \frac{(8r^3+25r^2+4r-28)^\frac{1}{2}}{6}+\frac{r}{2}+1$ --- the lower bound of $n\geq r^{3/2}+\frac{r}{2}+1$ is used in the statement simply for presentation.
\end{rk}

One could continue to play this game, obtaining explicit formulae for different ranges of $n$, however the calculations quickly become increasingly complicated.


\begin{thebibliography}{99}
\bibitem{beh}
S. Behrens, C. Erbes, M. Ferrara, S.G. Hartke, B. Reiniger, H. Spinoza, C. Tomlinson, \emph{New results on degree sequences of uniform hypergraphs}, Electron. J. Combin. \textbf{20(4)} (2013), \#P14.

\bibitem{blaha}
K.D. Blaha, \emph{Minimum bases for permutation groups: the greedy approximation}, J. Algorithms \textbf{13(2)}
(1992), 297--306.

\bibitem{burn2}
T.C. Burness, R.M. Guralnick, J. Saxl, \emph{On base sizes for symmetric groups}, Bull. Lond. Math. Soc. \textbf{43(2)} (2011), 386--391.

\bibitem{burn1}
T.C. Burness, M.W. Liebeck, A. Shalev, \emph{Base sizes for simple groups and a conjecture of Cameron}, Proc. Lond. Math. Soc. (3) \textbf{98(1)} (2009), 116--162.

\bibitem{ck}
 P. J. Cameron, W. M. Kantor, \emph{Random permutations: some group-theoretic aspects}, Combin.
Probab. Comput. \textbf{2} (1993), 257--262.

\bibitem{hal}
Z. Halasi, \emph{On the base size for the symmetric group acting on subsets}, Studia Sci. Math. Hungar. \textbf{49(4)} (2012),
492--500.

\bibitem{lie3}
M.W. Liebeck, A. Shalev, \emph{Simple groups, permutation groups, and probability}, J. Amer. Math. Soc. \textbf{12(2)}, (1999), 497--520.

\bibitem{MeSp}
G. Mecenero, P. Spiga, \emph{A formula for the base size of the symmetric group in its action on subsets}, arXiv preprint: \url{https://arxiv.org/abs/2308.02337}, submitted 04/08/2023.

\bibitem{ms}
J. Morris, P. Spiga, \emph{On the base size of the symmetric and the alternating group acting on partitions}, J. Algebra \textbf{587} (2021), 569--593.

\bibitem{ser}
\'A. Seress, \emph{Permutation group algorithms}, Cambridge Tracts in Mathematics \textbf{152}, Cambridge University Press
(2003).
\end{thebibliography}
\end{document}